\documentclass[letterpaper,12pt]{article}

\usepackage[bookmarks=false]{hyperref}

\hypersetup{colorlinks,breaklinks,
            linkcolor=black,urlcolor=blue,
            anchorcolor=black,citecolor=black}

\usepackage{amsmath}
\usepackage{amsfonts}
\usepackage{amssymb}
\usepackage{amsthm}
\usepackage{mathrsfs}
\usepackage{theoremref}
\usepackage[shortlabels]{enumitem}
\usepackage{makecell}
\usepackage[strict]{changepage}

\newcommand{\cC}{\mathcal{C}}

\usepackage{color}

\author{Hy Ginsberg}
\date{}
\title{Totally Symmetric Quasigroups of Order 16}

\begin{document}

\setcounter{secnumdepth}{2}
\numberwithin{equation}{section}

\newtheorem{main}{Theorem}

\maketitle

\begin{abstract}
\noindent
We present the number of totally symmetric quasigroups (equivalently, totally symmetric Latin squares) of order 16,
as well as the number of isomorphism classes of such objects.
Totally symmetric quasigroups of orders up to and including 16 that are (respectively) medial, idempotent,
and unipotent are also enumerated.
\end{abstract}

\bigskip

Suppose $xy = z$ for some elements $x,y,z$ in a quasigroup $Q$ of order $n$.  We say that $Q$ is
\emph{totally symmetric} if this implies that all six equations obtained by permuting the symbols $x$, $y$, and $z$ in this equation hold: 
$xy = z$, $xz = y$, $yx = z$, $yz = x$, $zx = y$, and $zy = x$.

In \cite{baileyenum,baileyenumx} Rosemary Bailey enumerated the totally symmetric quasigroups of orders up to and including $n=10$; these results were
extended by Brendan McKay and Ian Wanless through $n=15$ in \cite{mckaywanless}.  The main purpose of this paper is to announce the results 
for $n=16$, which are as follows:
\begin{main}
There are 
\begin{center}
91,361,407,076,595,590,705,971,200
\end{center}
totally symmetric quasigroups of order 16; these are divided into
\begin{center}
4,366,600,209,354
\end{center}
isomorphism classes.  
\end{main}
\noindent

We give more detailed information in Table 1, which subdivides these totals based on the orders of the automorphism
groups of the quasigroups in each class.  Each entry in the table specifies an order of an automorphism group,
the number of isomorphism classes of order 16 totally symmetric quasigroups having an automorphism group of that order, and the total number
of quasigroups contained in those isomorphism classes.

A brief overview of the algorithm used to calculate these results is given in the ``Procedural Summary'' section of this document.

\begin{table}[h]
\begin{center}
\caption{Totally Symmetric Quasigroups of Order 16}
\begin{tabular}{|r|r|r|}
\hline
\multicolumn{1}{|c|}{$|\text{Aut}|$} & Isomorphism Classes & \multicolumn{1}{c|}{Quasigroups} \\
\hline
1 & 4,366,595,344,552 & 91,361,356,919,980,461,490,176,000 \\
2 & 4,706,496 & 49,236,513,458,356,224,000 \\
3 & 60,959 & 425,144,116,260,864,000 \\
4 & 92,028 & 481,370,626,953,216,000 \\
5 & 78 & 326,395,522,252,800 \\
6 & 564 & 1,966,742,249,472,000 \\
7 & 36 & 107,602,919,424,000 \\
8 & 3,986 & 10,424,780,061,696,000 \\
12 & 84 & 146,459,529,216,000 \\
16 & 363 & 474,685,795,584,000 \\
20 & 2 & 2,092,278,988,800 \\
21 & 3 & 2,988,969,984,000 \\
24 & 76 & 66,255,501,312,000 \\
32 & 92 & 60,153,020,928,000 \\
36 & 1 & 581,188,608,000 \\
48 & 7 & 3,051,240,192,000 \\
60 & 1 & 348,713,164,800 \\
64 & 6 & 1,961,511,552,000 \\
96 & 8 & 1,743,565,824,000 \\
168 & 1 & 124,540,416,000 \\
192 & 8 & 871,782,912,000 \\
288 & 1 & 72,648,576,000 \\
1,344 & 1 & 15,567,552,000 \\
20,160 & 1 & 1,037,836,800 \\
\hline
\end{tabular}
\end{center}
\end{table}

\noindent
(We shall have more to say about the unique class with the largest automorphism group in the ``Sample Results'' section.)

Returning to our totally symmetric quasigroup $Q$ of order $n$, 
if $$w(x(yz)) = y(x(wz))$$ for all $w,x,y,z \in Q$, then $Q$ is said to be \emph{medial}.  
This property can also be expressed as $(wx)(yz) = (wy)(xz)$,
and has been referred to by various authors as \emph{abelian} \cite{murdoch,bruck,schwenk} or \emph{entropic} \cite{etherington};
\emph{medial} appears to be the most modernly accepted term \cite{marczak,shcherbacov,stanovsky,byoung}, and so has been adopted for this paper.   

The interest in medial totally symmetric quasigroups arises naturally in the study of elliptic curves, where one can
define an addition for points on a curve by fixing a point $p$ (generally taken to be the projective point at 
infinity) and defining
$$x + y = p(xy).$$
Under this definition, if $Q$ is totally symmetric, then $p$ is the additive identity, $-x = px$, and the
addition thus defined is associative \emph{if and only if} $Q$ is medial
(see the beautiful exercise 1.11 in \cite{ratpoints}, which inspired the author's investigation into these objects).
This addition enjoys the useful property that 
the sum of any two rational points on an elliptic curve $\cC$ is itself a rational point on $\cC$.

Recent results of Benjamin Young \cite{byoung} establish that the number of medial totally symmetric quasigroups of order $n$ is precisely the number
of ``labeled'' abelian groups of order $n$ (i.e.\ counting isomorphic but non-identical groups separately); this is sequence A034382 in the
\emph{On-Line Encyclopedia of Integer Sequences} \cite{oeis}.  Furthermore, J.\ Schwenk \cite{schwenk} gives a formula for the number of isomorphism classes
of medial totally symmetric quasigroups: If 3 does not divide $n$, it is precisely the number of isomorphism classes of abelian groups of order $n$;
otherwise each abelian group of order $n$ contributes one more than the number of non-isomorphic cyclic 3-groups in its invariant factor decomposition
(for $n \le 16$ with $n$ divisible by 3, each abelian group of order $n$  has a single such factor, so for the quasigroups considered in this paper 
there are in this case twice as many isomorphism classes of medial totally symmetric quasigroups as there are isomorphism classes of abelian groups).

Data on totally symmetric quasigroups and medial totally symmetric quasigroups of orders 1 through 16 is presented in Table 2, below.

\begin{table}[h]
\begin{adjustwidth}{-0.8in}{}
\begin{center}
\caption{Totally Symmetric Quasigroups}
\begin{tabular}{|r|r|r|r|r|}
\hline
                      & \multicolumn{2}{c|}{Totally Symmetric Quasigroups}                      & \multicolumn{2}{c|}{Medial}     \\
\thead{Order}         & \thead{Number}                                    & \thead{Classes}    & \thead{Number} & \thead{Classes}    \\
\hline
           1          &         1                                         &         1          &         1          & 1 \\
\hline
           2          &         2                                         &         1          &         2          & 1 \\
\hline
           3          &         3                                         &         2          &         3          & 2 \\
\hline
           4          &        16                                         &         2          &        16          & 2 \\
\hline
           5          &        30                                         &         1          &        30          & 1 \\
\hline
           6          &       480                                         &         3          &       360          & 2 \\
\hline
           7          &      1290                                         &         3          &       840          & 1 \\
\hline
           8          &   163,200                                         &        13          &    15,360          & 3 \\
\hline
           9          &   471,240                                         &        12          &    68,040          & 4 \\
\hline
          10          &        386,400,000                                &        139         &   907,200          & 1 \\
\hline
          11          &       2,269,270,080                               &        65          &   3,991,680        & 1 \\
\hline
          12          &       12,238,171,545,600                          &      25,894        &   159,667,200      & 4 \\
\hline
          13          &      149,648,961,369,600                          &      24,316        &   518,918,400      & 1 \\
\hline
          14          &      8,089,070,513,113,497,600                    &   92,798,256       &    14,529,715,200  & 1 \\
\hline
          15          &  160,650,421,233,958,656,000                      &  122,859,802       &   163,459,296,000  & 2 \\
\hline
          16         & 91,361,407,076,595,590,705,971,200                 & 4,366,600,209,354  & 4,250,979,532,800  & 5 \\
\hline
\end{tabular}
\end{center}
\end{adjustwidth}
\end{table}

We include as well (in Table 3) the numbers of \emph{idempotent} and \emph{unipotent} totally symmetric quasigroups and classes, 
where a quasigroup is \emph{idempotent}
if $xx = x$ for all $x \in Q$, and \emph{unipotent} if $xx = k$ for all $x \in Q$ and some fixed $k \in Q$.
These properties are closely related; McKay and Wanless prove that there is a bijective correspondence between isomorphism classes 
of idempotent totally symmetric quasigroups of order $n$ and isomorphism classes of unipotent totally symmetric quasigroups 
of order $n+1$ \cite[Theorem 5.2]{mckaywanless}, and that moreover the number of
unipotent totally symmetric quasigroups of order $n+1$ is precisely $n+1$ times the number of
idempotent totally symmetric quasigroups of order $n$ \cite[Theorem 2.4(vi) and Lemma 3.1]{mckaywanless}.

\begin{table}[h]
\begin{center}
\caption{Unipotent and Idempotent Totally Symmetric Quasigroups}
\begin{tabular}{|r|r|r|r|r|}
\hline
                      & \multicolumn{2}{c|}{Idempotent}      & \multicolumn{2}{c|}{Unipotent} \\
\thead{Order}         & \thead{Number}     & \thead{Classes} & \thead{Number}  & \thead{Classes} \\
\hline
           1          &        1           & 1               &        1            & 1 \\
\hline
           2          &        0           & 0               &        2            & 1 \\
\hline
           3          &        1           & 1               &        0            & 0 \\
\hline
           4          &        0           & 0               &        4            & 1 \\
\hline
           5          &        0           & 0               &        0            & 0 \\
\hline
           6          &        0           & 0               &        0            & 0 \\
\hline
           7          &       30           & 1               &        0            & 0 \\
\hline
           8          &        0           & 0               &      240            & 1 \\
\hline
           9          &      840           & 1               &        0            & 0 \\
\hline
          10          &        0           & 0               &     8,400           & 1 \\
\hline
          11          &        0           & 0               &        0            & 0 \\
\hline
          12          &        0           & 0               &        0            & 0 \\
\hline
          13          &     1,197,504,000  & 2               &        0            & 0 \\
\hline
          14          &        0           & 0               & 16,765,056,000      & 2 \\
\hline
          15          & 60,281,712,691,200 & 80              &        0            & 0 \\
\hline
          16          &        0           & 0               & 964,507,403,059,200 & 80 \\
\hline
\end{tabular}
\end{center}
\end{table}

The results of Bailey, McKay and Wanless, Young, and Schwenk all provide independent confirmation of the correctness of the results reported in Theorem 1,
as the software used to establish that theorem reproduces a great many of the results computed and predicted by these authors.  
Specifically, the software gives:
\begin{itemize}
\item[-] the number of totally symmetric quasigroups, as well as the number of isomorphism classes of totally symmetric quasigroups,
computed by McKay and Wanless for sets of order $n \leq 15$ (agreeing, of course, with Bailey's results for $n \le 10$);
\item[-] the number of medial totally symmetric quasigroups of order $n$ for all $n \leq 16$ predicted by Young; 
\item[-] the number of isomorphism classes of medial totally symmetric quasigroups of order $n$ for all $n \leq 16$ predicted by Schwenk;
\item[-] the number of idempotent totally symmetric quasigroups, as well as the number of isomorphism classes of idempotent totally symmetric
         quasigroups computed by McKay and Wanless for those orders $n < 16$ that have such objects ($n = 1, 3, 7, 9, 13, 15$).
\item[-] the number of unipotent totally symmetric quasigroups, as well as the number of isomorphism classes of unipotent totally symmetric
         quasigroups computed by McKay and Wanless for those orders $n \leq 16$ that have such objects ($n = 1, 2, 4, 8, 10, 14, 16$).
\end{itemize}

\subsection*{Procedural Summary}

Viewing a totally symmetric quasigroup as a Latin square, or, equivalently, as being represented by its Cayley table,
the entries along the main diagonal are of precisely two types -- those for which $xx = x$ (i.e.\ idempotent elements)
and those for which this is not the case.  
In \cite{baileyenum}, Rosemary Bailey presented criteria constraining the quantity of each type of diagonal entry
based on the order $n$ of the quasigroup, and further described a procedure for constructing directed graphs for each
permissible configuration of the diagonal.  The software developed to enumerate quasigroups for this paper begins by
constructing these ``Bailey graphs''; for $n=16$ there are 980 distinct, non-isomorphic Bailey graphs: 901 with one idempotent element, 
77 with four idempotent elements, and 2 with seven idempotent elements.  These correspond to 980 starting configurations, 
each with a distinct, non-isomorphic arrangement of the elements on the diagonal, and -- in deference to the magnitude
of the problem -- each is processed separately, with the results recorded and ultimately combined.

The fundamental process amounts to constructing totally symmetric quasigroups (by building their Cayley tables) and then checking completed
instances for isomorphism to quasigroups already constructed.  In \cite{mckaywanless}, McKay and Wanless describe a technique for
extending a Bailey graph to a directed graph that completely describes the quasigroup; we employ this technique, allowing us to use the 
\verb|nauty| graph automorphism software \cite{nauty} to test for isomorphism.

A great deal of technical sleight-of-hand is required to make the algorithm workable in practice.
Although there is a wide range of complexity among the starting configurations, 
a fairly typical example might have approximately 25 billion isomorphism classes and 
500 sextillion quasigroups.
Maintaining data structures for individual quasigroups is thankfully not necessary,
but maintaining 25 billion isomorphism classes -- with enough information to distinguish
them from each other -- is a daunting task, and in fact cannot be done within any reasonable limit on the amount of computer memory.
The McKay-Wanless graph for $n=16$ requires 408 bytes to represent in \verb|nauty|; using compression to approximately halve this 
amount still results in an isomorphism class structure that requires roughly 240 bytes.
This amounts to over 6 \emph{terabytes} of memory just to store the isomorphism classes; naturally there are other memory
requirements as well, making the prospect unfeasible for even the most powerful of today's readily available consumer 
computers.  This limitation is addressed by maintaining a large hash table of isomorphism classes, and, when an appropriate 
memory threshold is reached, flushing sections of the hash table to disk for storage and later processing.  As the process
is ongoing, this then requires further flushing of newly constructed isomorphism classes to disk, 
should they happen to belong to already flushed sections of the hash table.  
In some of the more challenging configurations the available disk space proves to be insufficient for the task,
leaving no option but to discard some of the flushed sections of the hash table; when this occurs subsequent iterations
through the entire process are required to recover the abandoned data.

\subsection*{Sample Results}

We include in this section further details pertaining to a sampling of some of the more symmetric starting configurations -- in general 
the more symmetric configurations are easier to compute, as they lend themselves to a higher degree of isomorph detection early on in 
the process.  These examples would be good candidates for independent verification of the results presented herein, without committing 
to the full computation.

As described in the previous section, each starting configuration includes precisely 1, 4, or 7 idempotent elements; the highest degree
of symmetry is found in some of the configurations with 1 idempotent element, and these have the added advantage of having more of the
Cayley table specified from the start (the idempotent elements contribute no additional information to the configuration, whereas total
symmetry implies that if $xx = y$ for some $x \neq y$, then $xy = x$ and $yx = x$, yielding two additional table entries).  Thus the simplest
configurations to compute are those that are highly symmetric with a single idempotent element.  We present four such configurations (those
with the largest automorphism groups -- i.e.\ the most symmetric), as well as the most symmetric configuration with 4 idempotent elements,
and the more symmetric of the two configurations with 7 idempotent elements.
For each we specify the initial configuration of the main diagonal as a 16-tuple, with the $i^\text{th}$ coordinate 
corresponding to the $i^\text{th}$ diagonal entry of the Cayley table (the result of the product of $i$ with itself).
We follow the convention prevalent in Latin squares of using the symbols 1 -- 16.

\subsubsection{Example 1 (The unipotent case)}
Diagonal: (1 \ 1 \ 1 \ 1 \ 1 \ 1 \ 1 \ 1 \ 1 \ 1 \ 1 \ 1 \ 1 \ 1 \ 1 \ 1)

\noindent
Order of the automorphism group of the starting configuration: 1,307,674,368,000 

\begin{center}
\begin{tabular}{|r|r|r|}
\hline
\multicolumn{1}{|c|}{$|\text{Aut}|$} & Isomorphism Classes & \multicolumn{1}{c|}{Quasigroups} \\
\hline
1 & 36 & 753,220,435,968,000 \\
2 & 6 & 62,768,369,664,000 \\
3 & 12 & 83,691,159,552,000 \\
4 & 8 & 41,845,579,776,000 \\
5 & 1 & 4,184,557,977,600 \\
6 & 1 & 3,487,131,648,000 \\
8 & 2 & 5,230,697,472,000 \\
12 & 3 & 5,230,697,472,000 \\
21 & 1 & 996,323,328,000 \\
24 & 2 & 1,743,565,824,000 \\
32 & 1 & 653,837,184,000 \\
36 & 1 & 581,188,608,000 \\
60 & 1 & 348,713,164,800 \\
96 & 1 & 217,945,728,000 \\
168 & 1 & 124,540,416,000 \\
192 & 1 & 108,972,864,000 \\
288 & 1 & 72,648,576,000 \\
20,160 & 1 & 1,037,836,800 \\
\hline
\end{tabular}
\end{center}

\noindent
Total number of isomorphism classes: 80

\noindent
Total number of quasigroups: 964,507,403,059,200

\noindent
Note that the last class, containing 1,037,836,800 quasigroups, is medial as well as unipotent, is the only class with
both of these properties, and has the largest order automorphism group of any 
class (see Table 1).

\subsubsection{Example 2}    
Diagonal: (1 \ 1 \ 2 \ 2 \ 2 \ 2 \ 2 \ 2 \ 2 \ 2 \ 2 \ 2 \ 2 \ 2 \ 2 \ 2)

\noindent
Order of the automorphism group of the starting configuration: 87,178,291,200

\begin{center}
\begin{tabular}{|r|r|r|}
\hline
\multicolumn{1}{|c|}{$|\text{Aut}|$} & Isomorphism Classes & \multicolumn{1}{c|}{Quasigroups} \\
\hline
1 & 649 & 13,578,890,637,312,000 \\
2 & 40 & 418,455,797,760,000 \\
3 & 37 & 258,047,741,952,000 \\
4 & 31 & 162,151,621,632,000 \\
6 & 2 & 6,974,263,296,000 \\
8 & 10 & 26,153,487,360,000 \\
12 & 4 & 6,974,263,296,000 \\
16 & 1 & 1,307,674,368,000 \\
21 & 2 & 1,992,646,656,000 \\
24 & 5 & 4,358,914,560,000 \\
32 & 3 & 1,961,511,552,000 \\
96 & 1 & 217,945,728,000 \\
192 & 1 & 108,972,864,000 \\
1,344 & 1 & 15,567,552,000 \\
\hline
\end{tabular}
\end{center}

\noindent
Total number of isomorphism classes: 787

\noindent
Total number of quasigroups: 14,467,611,045,888,000

\subsubsection{Examples 3 and 4}    

These two examples, having non-isomorphic starting configurations whose automorphism groups have the same order,
yield precisely the same results, and so are presented together.

\noindent
Diagonal (Example 3): (1 \ 1 \ 2 \ 2 \ 2 \ 2 \ 2 \ 2 \ 2 \ 2 \ 2 \ 2 \ 2 \ 2 \ 1 \ 1) \\
Diagonal (Example 4): (1 \ 1 \ 1 \ 1 \ 1 \ 1 \ 1 \ 1 \ 1 \ 1 \ 1 \ 1 \ 1 \ 1 \ 2 \ 2)

\noindent
Order of the automorphism group of the starting configuration: 958,003,200

\begin{center}
\begin{tabular}{|r|r|r|}
\hline
\multicolumn{1}{|c|}{$|\text{Aut}|$} & Isomorphism Classes & \multicolumn{1}{c|}{Quasigroups} \\
\hline
1 & 4,714 & 98,630,031,532,032,000 \\
2 & 177 & 1,851,666,905,088,000 \\
3 & 42 & 292,919,058,432,000 \\
4 & 75 & 392,302,310,400,000 \\
6 & 6 & 20,922,789,888,000 \\
8 & 28 & 73,229,764,608,000 \\
12 & 1 & 1,743,565,824,000 \\
16 & 5 & 6,538,371,840,000 \\
24 & 1 & 871,782,912,000 \\
32 & 3 & 1,961,511,552,000 \\
48 & 1 & 435,891,456,000 \\
64 & 1 & 326,918,592,000 \\
96 & 1 & 217,945,728,000 \\
192 & 1 & 108,972,864,000 \\
\hline
\end{tabular}
\end{center}

\noindent
Total number of isomorphism classes: 5,056

\noindent
Total number of quasigroups: 101,273,277,321,216,000

\subsubsection{Example 5 (4 idempotent elements)}  
Diagonal: (1 \ 1 \ 3 \ 3 \ 5 \ 5 \ 7 \ 7 \ 7 \ 7 \ 7 \ 7 \ 7 \ 7 \ 7 \ 7)

\noindent
Order of the automorphism group of the starting configuration: 2,177,280

\begin{center}
\begin{tabular}{|r|r|r|}
\hline
\multicolumn{1}{|c|}{$|\text{Aut}|$} & Isomorphism Classes & \multicolumn{1}{c|}{Quasigroups} \\
\hline
1 & 1,487,202 & 31,116,414,967,013,376,000 \\
2 & 840 & 8,787,571,752,960,000 \\
3 & 242 & 1,687,771,717,632,000 \\
4 & 698 & 3,651,026,835,456,000 \\
8 & 38 & 99,383,251,968,000 \\
12 & 6 & 10,461,394,944,000 \\
24 & 6 & 5,230,697,472,000 \\
\hline
\end{tabular}
\end{center}

\noindent
Total number of isomorphism classes: 1,489,032

\noindent
Total number of quasigroups: 31,130,656,412,663,808,000

\subsubsection{Example 6 (7 idempotent elements)}  
Diagonal: (1 \ 1 \ 3 \ 3 \ 5 \ 5 \ 7 \ 7 \ 9 \ 9 \ 11 \ 11 \ 13 \ 13 \ 13 \ 13)

\noindent
Order of the automorphism group of the starting configuration: 4,320

\begin{center}
\begin{tabular}{|r|r|r|}
\hline
\multicolumn{1}{|c|}{$|\text{Aut}|$} & Isomorphism Classes & \multicolumn{1}{c|}{Quasigroups} \\
\hline
1 & 3,055,198,131 & 63,923,268,561,123,299,328,000 \\
2 & 63,312 & 662,331,836,694,528,000 \\
3 & 3,233 & 22,547,793,235,968,000 \\
4 & 372 & 1,945,819,459,584,000 \\
6 & 104 & 362,661,691,392,000 \\
12 & 12 & 20,922,789,888,000 \\
\hline
\end{tabular}
\end{center}

\noindent
Total number of isomorphism classes: 3,055,265,164

\noindent
Total number of quasigroups: 63,923,955,770,157,170,688,000

\subsection*{Hardware}

The project was completed on a pair of Dell Precision 5810 workstations (running Fedora Linux at runlevel 3);
both computers were equipped with 256 GB of RAM and 12 terabytes
of hard drive space (comprised of three 4 terabyte drives, configured in a RAID0 array).  One of the computers had an additional 1 TB solid state
drive; that computer was outfitted with an Intel Xeon E5-2699A v4 22-core 2.4 GHz processor; 
the other had an Intel Xeon E5-2697A v4 16-core 2.6 GHz processor.  (These specifications reflect the ultimate configurations of
the two systems; there was initially only one, and much upgrading of components -- including CPUs, memory, and disk drives -- 
was performed while the project was ongoing.)
The final results reported herein were obtained after approximately 12 months of computing.

\subsection*{Feasibility of Computing Results for Order 17}

Whereas there are 980 non-isomorphic starting configurations for totally symmetric quasigroups of order 16, there are only 35 such
configurations for order 17 (18 with 1 idempotent element, 9 with 4 idempotent elements, 5 with 7 idempotent elements, 
2 with 10 idempotent elements, and 1 with 13 idempotent elements).  This relative paucity of cases makes attempting the enumeration
of totally symmetric quasigroups of order 17 a seemingly attractive proposition.  Preliminary investigations, however, indicate that 
the problem is more difficult than it might initially seem.

The most symmetric of the order 17 starting configurations having one idempotent element (which should in principle be the
easiest configuration to compute) has an automorphism group of order 7,776.  This configuration completes in 27.43 hours.
For comparison, a sampling of order 16 configurations with a single idempotent element and comparable automorphism group orders
(5,760 and 8,640) all complete in between 3 and 6 \emph{minutes}.

The most symmetric of the order 17 starting configurations with 4 idempotent elements has an automorphism group of order 15,552;
it required four passes (because it used the entirety of the 12 terabytes of disk space on each of the first three) and took a total of 551 hours (23 days)
to complete.  For comparison, an order 16 case with 4 idempotent elements whose starting configuration has an automorphism group of order 8,640  
completes in 31 minutes.

The implication is that, although there are far fewer starting configurations for order 17, they are \emph{at least} 
two orders of magnitude more difficult to compute.  The examples given were presumably the simplest for 1 and 4 idempotents, 
respectively; the remainder of the configurations can
be expected to be significantly more challenging.  Moreover, as the sizes of the automorphism groups of the starting configurations decrease,
the increase in complexity is far from linear -- otherwise the results for order 16 would have been obtained much more quickly.
In my estimation, if the totally symmetric quasigroups of order 17 are to be enumerated, it is going to require much more powerful
hardware resources than the two workstations used for the order 16 computation, a great deal of time, and / or an algorithmic breakthrough.

\subsubsection*{Acknowledgments}
The insightful comments of the reviewers resulted in many substantial changes, most notably the inclusion of the information
in Table 1 and the sections on sample results and the feasibility of the order 17 computation; I am truly grateful for their suggestions.

The paper was also greatly improved through correspondence with Benjamin Young, who brought to my attention Schwenk's results on the 
number of isomorphism classes of medial totally symmetric quasigroups, thereby disembarrassing me of a 
simpler (and incorrect) conjecture that I would otherwise have submitted for publication.  Much obliged, Ben.

Heartfelt thanks to Jeff Robbins of LiveData Inc., who generously shared his unparalleled knowledge of modern computing platforms 
and techniques, helping design the fantasy system for which the funding unfortunately never materialized\ldots \ 
My remarkable brother-in-law, Robert Edelstein, a writer with no advanced mathematical training, 
somehow managed to beat me to the discovery of a body of published research on these objects\ldots \ 
His continued sincere interest provided a much needed outlet for my frequent, detailed, elaborate bursts of obsessive enthusiasm;
thank you, Rob.
Thanks as well to the brilliant Matthew Welz -- my ``mathematical brother'' -- for his assistance during the early stages of the work, 
and to our ``mathematical father,'' Richard Foote, without whom, as far as I am concerned, there would be no mathematics --
as far as I am concerned, Richard said ``Let $\epsilon > 0$ be given,'' and there was \emph{light}\ldots

\bibliographystyle{amsplain}
\bibliography{/home/hy/Dropbox/math/math}

\end{document}